\documentclass{conm-p-l}

\usepackage{amssymb,mathrsfs,mathscinet}
\usepackage{mathtools}
\usepackage{graphicx}

\usepackage{ifpdf}
\ifpdf 
    \usepackage[off]{pdfsync}
\fi
\newcommand{\ifpdfsyncstart}{\ifpdf \pdfsyncstart \fi}
\newcommand{\ifpdfsyncstop}{\ifpdf \pdfsyncstop \fi}
\usepackage{hyperref}

\newtheorem{thm}{Theorem}[section]

\newtheorem{cor}[thm]{Corollary}

\newtheorem{prop}[thm]{Proposition}
\newtheorem{conj}[thm]{Conjecture}
\newtheorem{question}[thm]{Question}
\newtheorem{prob}[thm]{Problem}

\theoremstyle{definition}
\newtheorem{exm}[thm]{Example}
\newtheorem*{exm:nonintqpCont}{%
    Example \ref{exm:nonintqp} Continued%
}%
\newtheorem{defn}[thm]{Definition}
\newtheorem{openproblem}[thm]{Open Problem}

\newcommand{\Z}{\mathbb{Z}}
\newcommand{\Q}{\mathbb{Q}}
\newcommand{\R}{\mathbb{R}}

\newcommand{\braces}[1]{\left\lbrace #1 \right\rbrace}

\newcommand{\D}{\mathcal{D}}
\DeclareMathOperator{\GL}{GL}
\newcommand{\SLZ}{\GL_{n}(\Z)}
\newcommand{\G}{\SLZ}
\DeclareMathOperator{\Aff}{Af{}f}

\DeclareMathOperator{\conv}{conv}

\newcommand{\Ehr}{\mathcal{L}}

\newlength{\negone}
\title[%
    Quasi-period Collapse
]{%
    Quasi-period Collapse and \( \operatorname{GL}_{n}(\mathbb{Z})
    \)-Scissors Congruence in Rational Polytopes
}%

\author{%
    Christian Haase
}%
\address{%
    Fachbereich Mathematik und Informatik \\
    Freie Universit\"at Berlin
}%
\email{%
    Christian.Haase@Math.FU-Berlin.de
}%
\thanks{%
    The first author is supported by Emmy Noether grant HA 4383/1
    of the German Research Foundation (DFG)%
}%

\author{%
    Tyrrell B. McAllister
}%
\address{%
    Wiskunde en Informatica \\
    Technische Universiteit Eindhoven
}%
\email{%
    tmcallis@win.tue.nl
}%
\thanks{%
    The second author is supported by the Netherlands Organisation
    for Scientific Research (NWO) Mathematics Cluster DIAMANT%
}%

\subjclass[2000]{}

\date{}

\begin{document}
    \begin{abstract}
        Quasi-period collapse occurs when the Ehrhart
        quasi-poly\-nomial of a rational polytope has a
        quasi-period less than the denominator of that polytope.
        This phenomenon is poorly understood, and all known cases
        in which it occurs have been proven with ad hoc methods.
        In this note, we present a conjectural explanation for
        quasi-period collapse in rational polytopes.  We show that
        this explanation applies to some previous cases appearing
        in the literature.  We also exhibit examples of Ehrhart
        polynomials of rational polytopes that are not the Ehrhart
        polynomials of any integral polytope.
        
        Our approach depends on the invariance of the Ehrhart
        quasi-poly\-nomial under the action of affine unimodular
        transformations.  Motivated by the similarity of this idea
        to the scissors congruence problem, we explore the
        development of a Dehn-like invariant for rational
        polytopes in the lattice setting.
    \end{abstract}

    \maketitle

    \section{Introduction}

    A \emph{convex rational} (respectively, \emph{integral})
    \emph{polytope} \( P \subset \R^{n} \) is the convex hull of
    finitely many points in \( \Q^{n} \) (respectively, \( \Z^{n}
    \)).  The \emph{dimension} of \( P \) is the dimension of the
    affine subspace of \( \R^{n} \) spanned by \( P \).  Dilating
    \( P \) by a positive integer factor \( k \) yields the
    polytope \( k P = \braces{x \in \R^{n} : \tfrac{1}{k} x \in P}
    \).  The \emph{denominator} of \( P \) is the minimum positive
    integer \( \D \) such that \( \D P \) is an integral polytope.
    A seminal result of Ehrhart in 1962 \cite{Ehr62} provides a
    beautiful description of the counting function giving the
    number \( \lvert kP \cap \Z^{n} \rvert \) of integer lattice
    points in \( k P \).
    \begin{thm}[\cite{Ehr62}]
    \label{thm:EhrhartQPs}
        If \( P \subset \R^{n} \) is a \( d \)-dimensional
        rational polytope, then \( \lvert kP \cap \Z^{n} \rvert \)
        is given by the restriction to the positive integers of a
        degree-\( d \) quasi-polynomial \( \Ehr_{P}: \Z \to \Z \).
        That is, there exist periodic functions \( c_{0}, \dotsc,
        c_{d} \colon \Z \to \Q \) such that \( c_{d} \) is not
        identically zero and
        \begin{equation*}
            \lvert kP \cap \Z^{n} \rvert
            =
            \Ehr_{P}(k)
            =
            c_{d}(k) k^{d} + \dotsb + c_{1}(k) k + c_{0}(k),
            \qquad
            k \in \Z_{> 0}.
        \end{equation*}         
    \end{thm}

    We call \( \Ehr_{P} \) the \emph{Ehrhart quasi-polynomial} of
    \( P \).  A positive integer \( N \) is a \emph{quasi-period}
    of \( \Ehr_{P} \) (or of \( P \)) if \( N \) is divisible by
    the periods of all of the coefficient functions \( c_{i} \),
    \( 0 \le i \le d \).  (We do not assume that \( N \) is the
    minimum such positive integer.)
    
    When \( P \) is an integral polytope, \( \Ehr_{P} \) has
    quasi-period 1; that is, \( \Ehr_{P}(k) \) is a polynomial
    function of \( k \).  More generally, the denominator \( \D \)
    of a polytope \( P \) is a quasi-period of \( \Ehr_{P} \)
    \cite{Ehr62}.  It is somewhat surprising that \( \D \) is not
    always the \emph{minimum} quasi-period of \( P \).  When the
    minimum quasi-period of \( P \) is less than \( \D \), we say
    that \emph{quasi-period collapse} has occurred.  Several
    important polyhedra appearing in the representation theory of
    Lie algebras exhibit period collapse, but the known proofs of
    these results are not given in terms of the polyhedral
    geometry \cite{DLM04, DLM06, DW02, KR86}.

    Quasi-period collapse cannot happen in dimension \( 1 \), but
    there exist families of polygons in \( \R^{2} \) with
    arbitrarily large denominators whose minimum quasi-periods are
    1.  This result was originally proved in \cite{MW05}, where
    the proof of polynomiality involved subdividing the polygons
    into polygonal pieces whose Ehrhart quasi-polynomials could be
    computed.  The periodic parts for these pieces could be seen
    by inspection to cancel, with the result that the counting
    function for the entire polygon was a polynomial.
    
    In this paper, we give a new approach to understanding
    quasi-period collapse in rational polytopes.  This approach
    yields a much simpler explanation for the polynomiality of the
    Ehrhart quasi-polynomials appearing in \cite{MW05} (see
    Example \ref{exm:MW05Triang} below).  The demonstration again
    depends upon polyhedral subdivisions.  However, instead of
    explicitly computing the Ehrhart quasi-polynomials of the
    pieces in this subdivision, we rearrange unimodular images of
    the pieces to form an integral polytope.  Since this
    rearrangement does not change the number of lattice points in
    the polytope or in any of its dilations, it follows
    immediately that the original Ehrhart quasi-polynomial is a
    polynomial.  Thus we avoid computing the Ehrhart
    quasi-polynomials of the individual pieces.
    
    This approach provides a unified framework for demonstrating
    quasi-period collapse of rational polytopes.  We conjecture
    that a polytope exhibits quasi-period collapse only when the
    pieces of some subdivision of the polytope can be rearranged
    by affine unimodular transformations to form a polyhedral
    complex with the ``right'' denominator.  See Conjecture
    \ref{conj:EhrhartPolysIffUnionOfIntegral} for a precise
    statement.  This motivates a study of the invariants of
    rational polyhedra under polyhedral subdivision and piecewise
    unimodular transformations.  This is reminiscent of the
    scissors congruence problem for the group of rigid motions in
    \( \R^{3} \).  In the classical scissors congruence problem,
    congruence classes of polyhedra are parameterized by volume
    and the Dehn invariant \cite{Syd65}.  This suggests that an
    analogous system of invariants might determine when two
    rational polyhedra are equidecomposable with respect to the
    group \( \Aff_n(\Z) \cong \G \ltimes \Z^{n} \) of affine
    unimodular transformations.

    \section{Proving polynomiality of Ehrhart quasi-polynomials}
    \label{sec:Examples}

    The phenomenon of quasi-period collapse for rational polytopes
    is in general poorly understood.  In this section, we give
    examples of rational polytopes that can be shown to have
    quasi-period 1 by subdivision and rearrangement of unimodular
    images of the pieces.  These examples serve to motivate the
    following section, in which we conjecture that this method
    applies to all examples of quasi-period collapse among
    rational polytopes.
    
    \begin{exm}    
    \label{exm:MW05Triang}
        Given an integer $\D \geq 2$, let $T$ be the triangle with
        vertices \( (0,0)^{t} \), \( (1,\frac{\D-1}{\D})^{t} \),
        and $(\D,0)^{t}$.  Subdivide \( T \) into two triangles by
        the line \( x = 1 \) (see left of Figure
        \ref{fig:TriangRearrange}).  Let \( L \) be the
        ``one-third-open" triangle strictly to the left of the
        line, and let \( R \) be the closed triangle to the right.
        Thus we have
        \begin{align*}
            L 
            & = \conv \{%
                    (0, 0)^{t},
                    (1, 0)^{t}, 
                    (1,\tfrac{\D-1}{\D})^{t} 
                \}
                \; \backslash \;
                [(1, 0)^{t}, (1,\tfrac{\D-1}{\D})^{t}] \\
            R
            & = \conv \{
                    (1, 0)^{t},
                    (\D, 0)^{t},
                    (1,\tfrac{\D-1}{\D})^{t}
                \}
        \end{align*}
        Let \( U \) be the affine unimodular transformation \(
        \R^{2} \to \R^{2} \) defined by
        \[%
            U(x)
            =
            \begin{bmatrix}
                \D-1 & -\D \\
                -1    & 1
            \end{bmatrix}
            x
            +
            \begin{bmatrix}
                1  \\
                1
            \end{bmatrix}.
        \]
        Then \( U(L) \) and \( R \) are disjoint, and their union
        is the \emph{integral} triangle \[ T' = \conv\{ (1,
        0)^{t}, (1, 1)^{t}, (\D, 0)^{t} \} \] (see right of Figure
        \ref{fig:TriangRearrange}).  By construction, \( \Ehr_{T'}
        = \Ehr_{T} \), and so, since \( T' \) is integral, \(
        \Ehr_{T} \) is a polynomial.
    \end{exm}
    
    The triangle in Example \ref{exm:MW05Triang} first appeared in
    \cite{MW05}, where it was used to establish the following
    theorem.
    
    \begin{thm}
    \label{thm:Ex}
        Given an integer $\D \ge 2$, there exists a polygon with
        denominator $\D$ whose Ehrhart quasi-polynomial is a
        polynomial.
    \end{thm}
    
    \ifpdfsyncstop
    \begin{figure}[tbp]
        \includegraphics{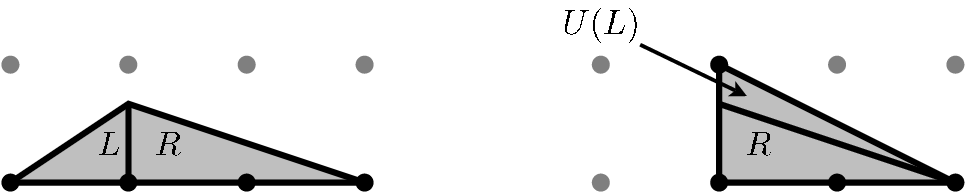}
        \caption{%
            Triangle \( T \) in the case \( \D = 3 \) on left, and
            \( \G \)-equi\-decom\-posable integral triangle on
            right
        }%
        \label{fig:TriangRearrange}
    \end{figure}
    \ifpdfsyncstart

    \begin{exm}
    \label{exm:StanleyPyramid}
        In \cite{Sta97}, Stanley gives an example of a \( 3
        \)-dimensional non-integral polyhedron with quasi-period
        1.  Let \( P \subset \R^{3} \) be the convex hull of the
        points \( (0, 0, 0)^{t} \), \( (1, 0, 0)^{t} \), \( (1, 1,
        0)^{t} \), \( (0, 1, 0)^{t} \), and \( (1/2, 0, 1/2)^{t}
        \).  This is the pyramid pictured on the left side of
        Figure \ref{fig:StanleyPyramid}.  To see that \( \Ehr_{P}
        \) is a polynomial, dissect \( P \) by the plane
        perpendicular to the vector \( w = (-1, 1, 1)^{t} \).  The
        intersection of this plane with \( P \) is indicated by
        the dark gray triangle in Figure \ref{fig:StanleyPyramid}.
        Let \( U \) be the unimodular transformation of \( \R^{3}
        \) whose matrix with respect to the standard basis is
        \begin{equation*}
            \begin{bmatrix*}[r]
                 1 &                             0 &  0  \\
                 1 &                             0 & -1  \\
                -1 & \makebox[\negone][r]{\( 1 \)} &  2
            \end{bmatrix*}.
        \end{equation*}
        Applying this transformation to the half-space \(
        \braces{x \in \R^{3} : w \cdot x \ge 0}\) maps \( P \) to
        the integral simplex on the right side of Figure
        \ref{fig:StanleyPyramid}.
    \end{exm}
    
    \ifpdfsyncstop
    \begin{figure}[tbp]
        \def\JPicScale{0.5}
        \includegraphics{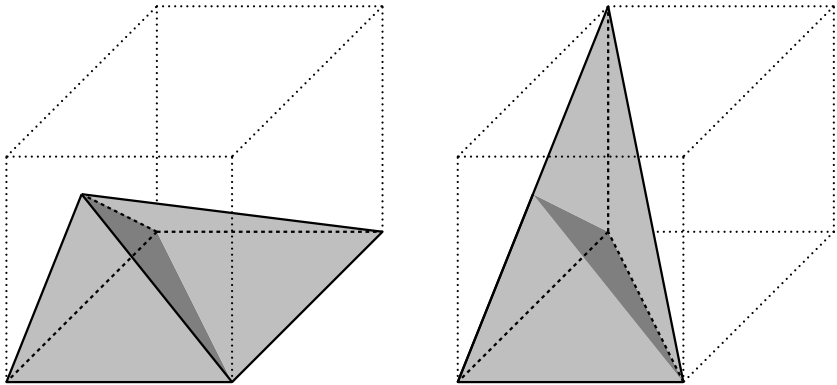}
        \caption{%
            Non-integral polyhedron and its integral image under
            piecewise unimodular transformation
        }%
        \label{fig:StanleyPyramid}
    \end{figure}
    \ifpdfsyncstart

    In the preceding examples, we showed that a non-integral
    polytope had a polynomial Ehrhart quasi-polynomial because it
    was, in some sense, a disguised integral polytope---it was an
    integral polytope up to rearrangement and unimodular
    transformation of its pieces.  One might be tempted to
    conjecture that all polytopes with polynomial Ehrhart
    quasi-polynomials are disguised integral polytopes in this
    sense.  In particular, this would imply that, for any rational
    polytope \( Q \), if \( \Ehr_{Q} \) is a polynomial, then \(
    \Ehr_{Q} = \Ehr_{P} \) for some \emph{integral} polytope \( P
    \).  However, this turns out not to be the case.  There exist
    Ehrhart polynomials that are not the Ehrhart polynomials of
    any integral polytope.

    \begin{exm}
    \label{exm:nonintqp}
        Let \( T \) be the triangle from Example
        \ref{exm:MW05Triang}, and let \( Q \) be the quadrilateral
        that results from the union of \( T \) with its reflection
        about the \( x \)-axis.  Then \( \Ehr_{Q}(k) = 2
        \Ehr_{T}(k) - \D k - 1 \) (correcting for the
        double-counting of the points on the \( x \)-axis).
        Hence, \( \Ehr_{Q} \) is also a polynomial.  Yet we claim
        that \( \Ehr_{Q} \) is not the Ehrhart polynomial of any
        integral polygon.  This is because \( Q \) has only two
        lattice points on its boundary, so, by \cite[Theorem
        3.1]{MW05}, the coefficient of the linear term of \(
        \Ehr_{Q} \) is 1.  But any integral polygon \( P \) has at
        least three lattice points on its boundary, so, by Pick's
        theorem, the coefficient of the linear term of \( \Ehr_{P}
        \) is at least \( 3/2 \).
    \end{exm}

    \section{Conjectures}
    \label{sec:Conjectures}

    As seen in the example concluding the previous section, a
    polytope \( P \) may have quasi-period 1 and yet not be the
    result of rearranging unimodular images of the pieces of an
    integral polytope.  Therefore, a more flexible formulation of
    the process carried out in the preceding examples is necessary
    if we hope to find a general explanation for the phenomenon of
    quasi-period collapse.

    To this end, recall that a \emph{simplex} is the convex hull
    of a finite set of affinely independent points.  An \emph{open
    simplex} is the interior of a simplex with respect to the
    affine subspace that it spans.  We call an open simplex
    \emph{integral} if its closure is integral.  The function \(
    \Ehr_{S} \) counting the lattice points in integral dilations
    of a \( d \)-dimensional open simplex \( S \) satisfies a
    well-known reciprocity property: \( \Ehr_{S}(k) = (-1)^{d}
    \Ehr_{\bar{S}} (-k)\), where \( \bar{S} \) is the closure of
    \( S \) \cite{Ehr67}.  In particular, if \( S \) is an
    integral open simplex, then \( \Ehr_{S}(k) \) is a polynomial
    function of \( k \).

    \begin{exm:nonintqpCont}
        The quadrilateral \( Q \) is a disjoint union of \( T \)
        and the reflection about the \( x \)-axis of those points
        in \( T \) strictly above the \( x \)-axis.  As in Example
        \ref{exm:MW05Triang}, each of these two sets may in turn
        be partitioned into open simplices that, under suitable
        rearrangement by unimodular transformations, form a
        disjoint union of integral open simplices.
    \end{exm:nonintqpCont}

    Let \( \Aff_n(\Z) \cong \G \ltimes \Z^{n} \) be the group of
    affine unimodular transformations on \( \R^{n} \).  To make
    the process employed above precise, we define the notion of \(
    \G \)-equidecomposability.  This definition first appeared
    in~\cite[\S3.1]{Kan98}; it is analogous to the classical
    Euclidean notion of equidecomposability (see, e.g.,
    \cite[Chapter 7]{AZ04}).

    \begin{defn}
    \label{defn:G-Equidecomposability}
        We say that two subsets $P, Q \subset \R^n$ are \emph{\(
        \G \)-equi\-de\-com\-posable} if there are open simplices
        \( T_{1}, \dotsc, T_{r} \) and affine unimodular
        transformations \( U_{1}, \dotsc, U_{r} \in \Aff_n(\Z) \)
        such that
        \begin{equation*}
            P = \coprod_{i=1}^{r} T_{i}
            \quad \text{and} \quad
            Q = \coprod_{i=1}^{r} U_{i}(T_{i}).
        \end{equation*}
        (Here, \( \coprod \) indicates disjoint union.)
    \end{defn}
    
    \begin{conj}
    \label{conj:EhrhartPolysIffUnionOfIntegral}
        Suppose that \( P \) is a rational polytope with
        quasi-period 1.  Then there exists a disjoint union \( Q
        \) of integral open simplices such that \( P \) and \( Q
        \) are \( \G \)-equidecomposable.
    \end{conj}

    Conjecture \ref{conj:EhrhartPolysIffUnionOfIntegral} has a
    natural generalization to polytopes whose quasi-periods
    collapse to values larger than \( 1 \): if \( P \) has minimum
    quasi-period \( N \), we conjecture that \( P \) is \( \G
    \)-equidecomposable with a disjoint union of open simplices
    whose denominators are at most \( N \).

    The decompositions employed in Examples \ref{exm:MW05Triang}
    and \ref{exm:StanleyPyramid} were reasonably easy to find.
    However, a systematic method of finding such decompositions is
    obviously desirable if we hope to extend this approach to a
    general technique for proving polynomiality of Ehrhart
    quasi-polynomials.
    
    \begin{openproblem}
       Find a systematic and useful technique that, given a
       rational polytope \( P \) that is \( \G \)-equidecomposable
       with some integral polytope \( Q \), produces a
       decomposition \( \{T_{i}\} \) of \( P \) and a set of
       unimodular maps \( \{U_{i}\} \) as in Definition
       \ref{defn:G-Equidecomposability}.
    \end{openproblem}
    
    \section{%
        \texorpdfstring{%
            \( \G \)%
        }{%
            GL\_n(Z)%
        }-Scissors Congruence
    }
    \label{sec:GLZ-ScissorsCongruence}

    Another phenomenon that appeared in the examples from Section
    \ref{exm:MW05Triang} was the equality of the Ehrhart
    quasi-polynomials of two distinct polytopes.  We say that two
    rational polytopes \( P \) and \( Q \) are \emph{Ehrhart
    equivalent} if and only if their Ehrhart quasi-polynomials are
    equal.  Obviously, any two \( \G \)-equidecomposable polytopes
    are Ehrhart equivalent.  But what about the converse?  Suppose
    a rational polytope $Q$ has the same Ehrhart quasi-polynomial
    as a polytope $P$.  Are $P$ and $Q$ $\G$-equidecomposable?
    
    The answer is known to be ``yes'' in the case \( d=2 \)
    \cite[Theorem 1.3]{Gre93}.  An analogy with the scissors
    congruence problem suggests that this is no longer the case
    for \( d \ge 3 \).  Nonetheless, as we prove below, a weak
    version of the converse direction does hold (Proposition
    \ref{prop:EhrEquivImpliesWeaklyEquidecomp}).  We also propose
    an ansatz for a $\G$-Dehn invariant, based on a theorem for
    reflexive polygons.

    \begin{question}
    \label{question:EhrEquivIffEquidecomp}
        Are Ehr\-hart-equi\-va\-lent rational polytopes always
        $\G$-equi\-decom\-posable?
    \end{question}
    
    \subsection{Weak $\G$-scissors congruence}
    
    If we allow more general translations of the pieces in a
    decomposition of \( P \), we get weak scissors congruences.

    \begin{defn}
        Two rational polytopes $P, Q \subset \R^n$ are
        \emph{weakly $\G$-equi\-de\-com\-posable} if they can be
        decomposed into rational polytopes $P_1, \ldots, P_r$ and
        $Q_1, \ldots, Q_r$, respectively, such that $P_i$ is
        equivalent to $Q_i$ via $\G \ltimes \Q^n$.
    \end{defn}

    This is equivalent to saying that there is a factor $k \in
    \Z_{>0}$ such that $kP$ and $kQ$ are (ordinarily)
    $\G$-equidecomposable.

    Observe that the weak version of $\G$-equidecomposability does
    not imply that the Ehrhart quasi-poly\-no\-mi\-als agree
    everywhere.  Nonetheless, they will agree at infinitely many
    arguments.  Therefore, if two \emph{integral} polytopes are
    weakly $\G$-equi\-de\-com\-posable, then they must be Ehrhart
    equivalent.

    \begin{prop}
    \label{prop:EhrEquivImpliesWeaklyEquidecomp}
        Let $P$ and $Q$ be Ehrhart-equivalent rational polytopes.
        Then $P$ and $Q$ are weakly $\G$-equidecomposable.
    \end{prop}

    \begin{cor}
        Two integral polytopes are Ehrhart equivalent if and only
        if they are weakly $\G$-equidecomposable.
    \end{cor}

    \begin{proof}[%
        Proof of Proposition
        \ref{prop:EhrEquivImpliesWeaklyEquidecomp}%
    ]%
        By a famous theorem of Kempf et al., there is a positive
        integer \( N \) such that \( NP \) and \( NQ \) are both
        integral and admit unimodular triangulations---i.e.,
        triangulations whose simplices are $\Aff_n(\Z)$-equivalent
        to the standard simplex \cite{KKMSD73}.  It is well known
        that the Ehrhart polynomial of a polytope determines the
        $f$-vector of a unimodular triangulation of that polytope
        (see, e.g., \cite[Corollary 2.5]{Sta80}).  Hence, the
        triangulations of \( NP \) and \( NQ \) have the same
        $f$-vector, and all simplices of a given dimension are
        equivalent under $\Aff_n(\Z)$.  Therefore, the
        corresponding simplices of \( P \) and \( Q \) are
        equivalent under \( \G \ltimes \Q^d \).  The claim
        follows.
    \end{proof}
    
    \subsection{A $\G$-Dehn invariant?}
    
    For the classical scissors congruence problem in three
    dimensions, one uses rigid motions rather than lattice
    preserving transformations.  The volume and the Dehn invariant
    \begin{equation*}
        \operatorname{Dehn}(P) \ 
        = \ \sum_{\text{\( e \) an edge of \( P \)}}
            \text{length}(e)
            \otimes
            \text{angle}(e)
        \quad \in \quad
            \R
            \otimes_\Z
            \R / \Z \pi
    \end{equation*}
    provide a complete set of invariants.  That is,
    $3$-dimensional polytopes $P$ and $Q$ are scissors congruent
    if and only if they have the same volume and the same Dehn
    invariant.  The ``only if'' part is relatively easy to see
    (see~\cite[Chapter 7]{AZ04}), because the Dehn invariant is
    additive, and decompositions of polyhedra satisfy the
    following two properties.
    \begin{itemize}
        \item[$(\pi)$] 
        A decomposition edge through a two-dimensional face
        contributes an angle of $\pi$, so it does not contribute
        to the Dehn invariant.
        
        \item[$(2\pi)$] 
        A decomposition edge through the interior contributes an
        angle of $2\pi$, so it does not contribute to the Dehn
        invariant.
    \end{itemize}
    \begin{prob}
        Can we manufacture a Dehn-like invariant in the $\G$ case?
    \end{prob}
    This invariant, once constructed, will likely be more
    appropriate to detecting when two lattice polytopes are
    $\G$-equidecomposable into lattice polytopes, in particular,
    when unimodular triangulations exist.

    The role of the full circle $2\pi$ should be played by the
    ``12'' of Poonen and Rodriguez-Villegas~\cite{PRV00}.
    \begin{thm}
        The sum of the lengths of a reflexive polygon and its dual
        is $12$.
    \end{thm}
    Here, a lattice polygon is reflexive if it contains a unique
    interior lattice point, and the length is measured with
    respect to the lattice.  The polygon does not need to be
    convex.  In the non-convex case, the definition of the dual is
    a little harder~\cite{PRV00, HS04}.  Around a subdivision
    edge, we see a polygon with a distinguished interior
    point---the projection of the edge (see Figure
    \ref{fig:subdEdge}).
    
    \ifpdfsyncstop
    \begin{figure}[thb]
        \centering
        \includegraphics[width=40mm]{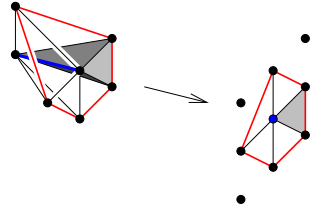}
        \caption{Projecting a subdivision edge}
        \label{fig:subdEdge}
    \end{figure}
    \ifpdfsyncstart
    
    This gives rise in a canonical way to a (possibly non-convex)
    reflexive polygon.  So we could mimic property $(2\pi)$ of the
    Dehn invariant by mapping to $\Z/12$.  Is there a way to
    incorporate the property $(\pi)$?
    
    
    \def\cprime{$'$}
\providecommand{\bysame}{\leavevmode\hbox to3em{\hrulefill}\thinspace}
\providecommand{\MR}{\relax\ifhmode\unskip\space\fi MR }
\providecommand{\MRhref}[2]{%
  \href{http://www.ams.org/mathscinet-getitem?mr=#1}{#2}
}
\providecommand{\href}[2]{#2}

\end{document}